\documentclass{article}
\usepackage{graphicx}
\usepackage[french, english]{babel}

\usepackage{amsmath}   
\usepackage{amssymb} 
\usepackage{amsthm}
\usepackage[hidelinks]{hyperref}
\usepackage[utf8]{inputenc}
\usepackage{tikz}
\usetikzlibrary{%
	cd, 
	matrix 
}

\usepackage{mathrsfs}
\usepackage[all]{xy}
\newcommand{\R}{\mathbb{R}}

\renewcommand{\d}{\partial}
\newcommand{\p}{\partial}

\newcommand{\gl}{\Lambda}
\newcommand{\gk}[1]{\gl_{#1}}
\newcommand{\gnot}{\gl_{0}}

\newcommand{\be}{\begin{equation}}
	\newcommand{\ee}{\end{equation}}
\newcommand{\ba}{\begin{aligned}}
	\newcommand{\ea}{\end{aligned}}
\newcommand{\N}{\mathbb{N}}
\newcommand{\fu}{\mathbf{u}}
\newcommand{\eps}{\varepsilon}
\newcommand{\Om}{\Omega}
\newcommand{\dd}{\:\mathrm{d}}
\newcommand{\Z}{\mathbb{Z}}

\newcommand{\cH}{\mathcal H}
\newcommand{\bu}{\bar{u}}

\newcommand{\reg}{\mathrm{reg}}
\newcommand{\sing}{\mathrm{sing}}
\newcommand{\ureg}{u_\mathrm{reg}}

\newcommand{\busing}{\bu_{\sing}}

\newcommand{\bfi}{\overline{f_i}}
\newcommand{\bfzero}{\overline{f_0}}
\newcommand{\bfun}{\overline{f_1}}
\newcommand{\bfj}{\overline{f_j}}

\newcommand{\qnot}{{Q^{0}}}
\newcommand{\qhalf}{{Q^{1/2}}}
\newcommand{\qone}{{Q^1}}

\newcommand{\supp}{\operatorname{supp}}

\newcommand{\hperp}{\mathcal{H}^\perp_{\operatorname{sg}}}

\newtheorem{theorem}{Theorem}

\newtheorem{remark}{Remark}[section]
\newtheorem{definition}{Definition}[section]
\newtheorem{lemma}{Lemma}[section]
\newtheorem{corollary}{Corollary}[section]
\newtheorem{proposition}{Proposition}[section]

\begin{document}
	
	%
	%
	\title{Nonlinear forward-backward problems}
	
	%
	%
\author{Anne-Laure Dalibard\thanks{Sorbonne Université, Université Paris Cité, CNRS, Laboratoire Jacques-Louis Lions, LJLL, F-75005 Paris, France, \texttt{anne-laure.dalibard@sorbonne-universite.fr}},
	Frédéric Marbach\thanks{ CNRS, École Normale Supérieure, Université PSL, Département de Mathématiques et applications, \texttt{frederic.marbach@ens.fr}},
	 Jean Rax\thanks{Sorbonne Université, Université Paris Cité, CNRS, Laboratoire Jacques-Louis Lions, LJLL, F-75005 Paris, France, \texttt{jean.rax@normalesup.org}}
	}

	%
	%
	\maketitle
	
	\begin{abstract}
	We prove the existence and uniqueness of strong solutions
	to the equation $u u_x - u_{yy} = f$ in the vicinity of the linear shear flow, subject to perturbations of the source term and lateral boundary conditions.
	Since the solutions we consider have opposite signs in the lower and upper half of the domain, this is a quasilinear forward-backward parabolic problem, which changes type across a critical curved line within the domain.
	In particular, lateral boundary conditions can be imposed only where the characteristics are inwards.
	There are several difficulties associated with this problem. First, the forward-backward geometry depends on the solution itself. This requires to be quite careful with the approximation procedure used to construct solutions.
	Second, and  more importantly, the linearized equations solved at each step of the iterative scheme admit a finite number of singular solutions, of which we provide an explicit construction.  
	This is similar to well-known phenomena in elliptic problems in nonsmooth domains.
	Hence, the solutions to the equation are regular if and only if the source terms satisfy a finite number of orthogonality conditions.
	A key difficulty of this work is to cope with these orthogonality conditions during the nonlinear fixed-point scheme.
	In particular, we are led to prove their stability with respect to the underlying base flow.
	To tackle this deceivingly simple problem, we develop a methodology which we believe to be both quite natural and adaptable to other situations in which one wishes to prove the existence of regular solutions to a nonlinear problem for suitable data despite the existence of singular solutions at the linear level.
	This paper is a shorter version of \cite{DMR}.

	\end{abstract}

\section{Introduction}
\label{sec:intro}

This paper is concerned with sign-changing solutions of 
the equation\begin{equation} \label{eq:eq0}
	u \d_x u - \d_{yy} u = f
\end{equation}
in the rectangular domain $\Omega := (x_0, x_1) \times (-1,1)$, where $f$ is an external source term.
A natural solution to \eqref{eq:eq0} with a null source term $\mathbf f= 0$ is the linear shear flow $\mathbf u(x,y) := y$, which changes sign across the horizontal line $\{ y = 0 \}$.
We are interested in strong solutions to~\eqref{eq:eq0} which are close (with respect to an appropriate norm) to this linear shear flow $\fu$.
Our purpose is to construct such solutions by perturbing the lateral boundary data at $x=x_0$ or $x=x_1$ or the source term $\mathbf{f} = 0$.

\begin{figure}[ht]
	\centering
	\includegraphics{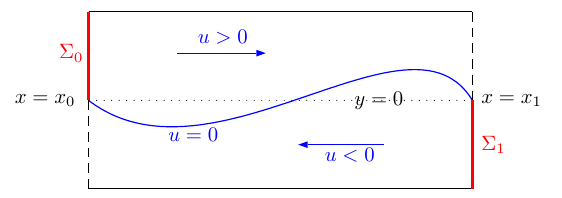}
	\caption{Fluid domain $\Omega$ and inflow boundaries $\Sigma_0 \cup \Sigma_1$}
	\label{fig:omega}
\end{figure}

Since such solutions will change sign across a line $\{ u = 0 \}$ lying within $\Omega$, a key feature of this work is that \eqref{eq:eq0} must be seen as a quasilinear forward-backward parabolic problem in the horizontal direction.
Thus, to ensure the existence of a solution, one must be particularly careful as to how one enforces the lateral perturbations.
More precisely, the problem is forward parabolic in the  domain above the line $\{u=0\}$, in which $u>0$, and therefore we shall  prescribe a boundary condition on $\Sigma_0 := \{x=x_0\}\cap\{u>0\}$; and backward parabolic in the domain below the line $\{u=0\}$, and we shall prescribe a boundary condition on $\Sigma_1 := \{x=x_1\}\cap \{u<0\}$.

Our purpose is to construct strong solutions of this system, in a high regularity functional space. However, one key difficulty of our work lies in the fact that, even when the source term $f$ is smooth, say in $C^\infty_0(\Om)$, \emph{solutions to \eqref{eq:eq0} have singularities in general}.
Actually, this feature is already present at the linear level, i.e.\ for the equation $y\d_x u -\d_{xx}u=f$.
For this linear equation, we prove that if $f$ is smooth, the associated weak solution to the linear system inherits the regularity of $f$ if and only if $f$ satisfies \emph{orthogonality conditions} (i.e.\ the scalar product of $f$ with some identified profiles must vanish). We also describe the singularities that appear when these orthogonality conditions are not satisfied.
In the nonlinear setting, the equivalent of these orthogonality conditions is the following (informal) statement: when $f$ is smooth and satisfies a  smallness assumption in a suitable functional space, strong solutions exist if and only if $f$ belongs to a finite codimensional space.

Due to the forward-backward nature of the problem, we must choose the lateral perturbations and the source term in a particular product space.
We therefore introduce the vector space
\begin{equation} \label{eq:E}
	\begin{split}
		\Big\{ (f,\delta_0,\delta_1) & \in
		C^\infty_c(\Omega)\times C^\infty([0,1]) \times C^\infty([-1,0])  ; \quad 
		\delta_i(0) = \p_y \delta_i(0) = \p_y^2 \delta_i(0) = 0 \\
		& \text{and } \delta_i((-1)^i) = \p_y^2 \delta_i((-1)^i) = 0
		\text{ for } i = 0, 1
		\Big\}
	\end{split}
\end{equation}
and $\mathcal{H}$, the Hilbert space defined as its completion with respect to the following norm, associated with the corresponding canonical scalar product
\begin{equation} \label{eq:H-norm}
	\| (f,\delta_0,\delta_1) \|_{\mathcal{H}}
	:= \|  f \|_{H^1_xH^1_y}  + \| \p_y^3 f \|_{L^2}
	+ \sum_{i \in \{0,1\}} \|\delta_i\|_{H^5} + \|(\p_y^2 \delta_i) / y \|_{H^1(|y| \dd y)}.
\end{equation}
For all $\eta>0$, we denote by $B_\eta$ the open ball of radius $\eta$ in $\cH$.

We establish the existence and uniqueness of solutions in the following anisotropic Sobolev space
\begin{equation}\label{def:Q1}
	\qone := L^2((x_0,x_1); H^5(-1,1)) \cap H^{5/3}((x_0,x_1);L^2(-1,1)).
\end{equation}
We first state a  result concerning the well-posedness in $\qone$ of the linear version of \eqref{eq:eq0} around the linear shear flow, \emph{up to two orthogonality conditions}:

\begin{theorem}[Orthogonality conditions for linear forward-backward parabolic equations] \label{theorem:existence-shear-X1}
	Let $\Sigma_0 := \{ x_0 \} \times (0,1)$ and $\Sigma_1 := \{ x_1 \} \times (-1, 0)$.
	There exists a vector subspace $\hperp \subset \mathcal{H}$ of codimension two such that, for each $(f,\delta_0, \delta_1) \in \mathcal{H}$, there exists a solution $u \in \qone$ to the problem
	\begin{equation}
		\label{eq:shear}
		\begin{cases}
			y \p_x u - \p_{yy} u = f,\\
			u_{\rvert \Sigma_i} = \delta_i, \\
			u_{\rvert y = \pm 1} = 0,
		\end{cases}
	\end{equation}
	if and only if $(f,\delta_0,\delta_1) \in \hperp$.
	Such a solution is unique and satisfies
	\begin{equation} \label{eq:u-qone}
		\|u\|_{\qone} \lesssim \| (f,\delta_0, \delta_1) \|_{\mathcal{H}}.
	\end{equation}
\end{theorem}

We then turn towards the nonlinear problem \eqref{eq:eq0}.
The main result of this paper is the following nonlinear generalization of Theorem  \ref{theorem:existence-shear-X1} for small enough perturbations.

\begin{theorem}[Existence and uniqueness of strong solutions to \eqref{eq:eq0} under orthogonality conditions] \label{theorem:main}
There exists $\eta>0$ and 	a Lipschitz submanifold $\mathcal{M}$ of $\mathcal{H}$ of codimension two, containing $0$ and included in $B_\eta$, such that, for every $(f,\delta_0, \delta_1) \in \mathcal{M}$, there exists a strong solution $u \in \qone$ to
	\begin{equation} \label{eq:yuuxuyy}
		\begin{cases}
			(y + u) \p_x u - \p_{yy} u = f,\\
			u_{\rvert \Sigma_i} = \delta_i, \\
			u_{\rvert y = \pm 1} = 0.
		\end{cases}
	\end{equation}
	Such solutions are unique in a small neighborhood of~$0$ in $\qone$ and satisfy the estimate~\eqref{eq:u-qone}.

	Conversely,    there exists $\eta > 0$ such that for all
	 $(f,\delta_0,\delta_1) \in B_\eta$, if there exists a solution 
	 $u \in \qone$  to \eqref{eq:yuuxuyy} such that $\|u\|_{\qone} \leq \eta$,
	then $(f,\delta_0, \delta_1) \in \mathcal{M}$.

\end{theorem}

\begin{remark}[About orthogonality conditions]
	The necessity for orthogonality conditions  in order to have higher regularity solutions is well known in the context of elliptic equations in domains with corners (see \cite{Grisvard1,Grisvard2}) and the situation is very similar here. 
	A misleading aspect is that it is quite easy, assuming the existence of a smooth solution of \eqref{eq:eq0}, to prove \emph{a priori} estimates at any order (see \cite{theseRax}).
	However, it can be proved (both in the setting of elliptic equations in domains with corners and in the framework of  Theorem \ref{theorem:existence-shear-X1}) that there exist $C^\infty_0$ source terms for which the unique weak solution is not smooth. In fact, we shall identify some explicit  profiles, which are singular in the vicinity of the points $(x_i, 0)$, and which are associated with smooth source terms.

\end{remark}

Our original motivation stems from fluid mechanics. Indeed, the stationary Prandtl equation, which describes the behavior of a fluid with small viscosity in the vicinity of a wall, reads
\be
\label{Prandtl}
\begin{cases}
	u \d_x u + v \d_y u - \d_{yy} u = - \d_x p_E, \\
	u_{|y=0} = v_{|y=0}=0,\\
	\lim_{y \to \infty } u(x,y)=u_E(x),
\end{cases}
\ee
where $u_E(x)$ (resp.\ $p_E(x)$) is the trace of an outer Euler flow (resp.\ pressure) on the wall, and satisfies $u_E \d_x u_E = - \d_x p_E$.

As long as $u$ remains positive, \eqref{Prandtl} can be seen as a nonlocal, nonlinear diffusion type equation, the variable $x$ being the evolution variable. Using this point of view, Oleinik (see e.g.\ \cite[Theorem 2.1.1]{MR1697762}) proved the local well-posedness of a solution to \eqref{Prandtl} when the equation \eqref{Prandtl} is supplemented with a boundary data $u_{|x=0}=u_0$, where $u_0(y)>0$ for $y>0$ and such that $u_0'(0)>0$. Let us mention that such positive solutions exist globally when $ \d_x p_E\leq 0$, but are only local when $\d_x p_E>0$. More precisely, when $\d_x p_E=1$ for instance, for a large class of boundary data $u_0$, there exists $x^*>0$ such that $\lim_{x\to x^*} \p_y u(x,0)=0$. Furthermore, the solution may develop a singularity at $x=x^*$, known as Goldstein singularity \cite{Goldstein,Stewartson,DM}.
Downstream of the singularity, the solution $u$ is expected to change sign, taking negative values in the vicinity of the boundary and positive values in the vicinity of $+\infty$, the usual convention being that $u_E(x)>0$.
Such solutions are called ``recirculating solutions'', and the zone where $u<0$ is called a recirculation bubble.

In the recent preprint \cite{IM2022}, Sameer Iyer and Nader Masmoudi prove \emph{a priori} estimates in high regularity norms for smooth solutions to the Prandtl equation \eqref{Prandtl}, in the vicinity of explicit self-similar recirculating flows, called Falkner-Skan profiles. The latter are given by
\begin{align}
	u(x,y) & = x^m f'(\zeta), \\
	v(x,y) & = - y^{-1} \zeta f(\zeta) - \frac{m-1}{m+1} y^{-1} \zeta^2 f'(\zeta),
\end{align}
where $\zeta := (\frac{m+1}{2})^{\frac 12} y x^{\frac{m-1}{2}}$ is the self-similarity variable, $m$ is a real parameter and $f$ is the solution to the Falkner-Skan equation
\begin{equation}
	f''' + f f'' + \beta (1- (f')^2) = 0,
\end{equation}
where $\beta = \frac{2m}{m+1}$, subject to the boundary conditions $f(0) = f'(0) = 0$ and $f'(+\infty) = 1$. Such flows correspond to an outer Euler velocity field $u_E(x) = x^m$. 
For some particular values of $m$ (or, equivalently, $\beta$), these formulas provide physical solutions to \eqref{Prandtl} which exhibit recirculation (see~\cite{MR0204011}). 
Obtaining \emph{a priori} estimates for recirculating solutions to the Prandtl system \eqref{Prandtl} is very difficult, due to the combination of several difficulties (forward-backward system, nonlocality of the transport term  $v\p_y u$, loss of derivative).

In the present paper, we have chosen to focus on a different type of difficulty, and to 
consider the toy-model \eqref{eq:eq0}, which differs from \eqref{Prandtl} through the lack of the nonlinear transport term $v \d_y u$ and its associated difficulties  and the exclusion of the zones close to the wall and far from the wall.
For the model \eqref{eq:eq0}, \emph{a priori} estimates are easy to derive, see \cite[Chapter 4]{theseRax}.
The difficulty lies elsewhere, as explained previously. 
Indeed, in order to construct a sequence of approximate solutions satisfying the \emph{a priori} estimates, we need to ensure that the orthogonality conditions are satisfied all along the sequence.
The core of the proof is to 
keep track of these orthogonality conditions, and to analyze their dependency on the sequence itself.
This strategy is presented in greater detail in an early version of our work \cite{DMR}, to which we will refer throughout the text.
For the Prandtl system \eqref{Prandtl}, this difficulty has  recently been tackled by Sameer Iyer and Nader Masmoudi in \cite{IM2023}, building upon their \emph{a priori} estimates of \cite{IM2022} and the ideas developed in \cite{DMR}.
The methodology developed in the present paper can be adapted to many other settings,  as explained in the latest version of our work, see \cite{DMRMEMS}.
A natural extension of \eqref{eq:eq0} is the 1d  nonlinear Fokker-Planck equation in a bounded domain: consider equations of the type
\be\label{eq:FP}
\begin{cases}
v\p_x g + b[g]\p_v g - a[g]\p_{vv} g=0, & (x,v)\in \Om,\\
g_{|\Sigma_i}=\delta_i,\\
g_{|v=\pm 1}=0.
\end{cases}
\ee
Note that in the above system, the vertical variable $z$ has been relabeled as $v$, as is customary in kinetic models.
The coefficients $a$ and $b$ are nonlocal functionals of the unknown $g$. 
A typical example (see \cite{Villani}) is
\[
v\p_x g = \rho^\alpha \p_v (T\p_v g + g(v-u))
\]
where $\alpha\in [0,1]$ and $\rho, u, T$ are respectively the local density, velocity and temperature, namely
\begin{align*}
	\rho (x)&=\int_\Om g(x,v)\dd v,\\
	\rho(x) u(x)&=\int_\Om v g(x,v)\dd v,\\
	\rho(x)|u(x)|^2 + \rho(x) T(x)&=\int_\Om |v|^2 g(x,v)\dd v.
		\end{align*}
We believe that under suitable assumptions on the operators $a$ and $b$, the methods presented here could allow us to prove the existence of strong solutions of \eqref{eq:FP} in the vicinity of smooth changing-sign solution.

\section{The case of the linear shear flow}
\label{sec:shear}

We begin with the study of the linear system \eqref{eq:shear}, for which we give a notion of weak solution:
\begin{definition}[Weak solution] \label{def:weak-shear}
	Let $f \in L^2((x_0,x_1);H^{-1}(-1,1))$.\\
		Let $\delta_0, \delta_1 \in L^2((-1,1), |y|\dd y)$. 
	We say that $u \in L^2((x_0,x_1);H^1_0(-1,1))$ is a  \emph{weak solution} to~\eqref{eq:shear} when, for all $v \in H^1(\Omega)$ vanishing on $\partial\Om \setminus (\Sigma_0\cup\Sigma_1)$, the following weak formulation holds
	\begin{equation} \label{weak-shear}
		- \int_\Omega y u \d_x v + \int_\Om \d_y u \d_y v
		= \int_\Omega f v  + \int_{\Sigma_0} y \delta_0 v - \int_{\Sigma_1} y \delta_1 v.
	\end{equation}
\end{definition}

Weak solutions in the above sense are known to exist since the work of Fichera \cite[Theorem XX]{MR0111931} (which concerns generalized versions of \eqref{eq:shear}, albeit with vanishing boundary data).
Uniqueness dates back to \cite[Proposition 2]{BG} by Baouendi and Grisvard. A natural question is then to consider strong solutions, i.e.\ solutions for which \eqref{eq:shear} holds almost everywhere.
The main result on this topic is due to Pagani \cite[Theorem~5.2]{Pagani2} (see also \cite{Pagani1}), who introduced the space
\[
Z^0:=\{ u\in L^2(\Om), y\p_x u\in L^2(\Om),\ \p_y u \in L^2 (\Om)\text{ and } \p_{yy} u \in L^2(\Om)\}.
\]
It can easily be proved (see \cite{DMR}) that $Z^0$ is continuously embedded in
\be\label{def:qnot}
\qnot:=H^{2/3}_x L^2_y\cap L^2_x H^2_y.
\ee

\begin{proposition}[\cite{Pagani2}] \label{p:pagani-shear}
	Let $\delta_0, \delta_1 \in H^1 ((-1,1), |y| \dd y)$ such that $\delta_0(1) = \delta_1(-1) = 0$.
	Let $f \in L^2(\Omega)$. 
	The unique weak solution $u$ to \eqref{eq:shear} belongs to $Z^0(\Omega)$ and satisfies
	\begin{equation} \label{estimate-pagani-Z0}
		\| u \|_{Z^0} \lesssim \| f \|_{L^2} + \| \delta_0 \|_{H^1 ((-1,1), |y| \dd y)} + \| \delta_1 \|_{H^1 ((-1,1), |y| \dd y)}.
	\end{equation}
\end{proposition}

Let us now observe that equation \eqref{eq:shear} is stable by differentiation with respect to $x$, at least formally. More precisely, if $u$ is a smooth (say, $H^1_x H^1_y$) solution of \eqref{eq:shear}, with smooth data $(f,\delta_0,\delta_1)$, then $\p_x u$ is also a solution of \eqref{eq:shear} with the data $(\p_x f, \Delta_0, \Delta_1)$, with
\be\label{eq:def-Di}
\Delta_i(y):= \frac{f(x_i,y) + \p_y^2 \delta_i(y)}{y}.
\ee
As  a consequence, \emph{if} $u\in H^1_x H^1_y$, then $u$ satisfies the a priori estimate
 \begin{equation} \label{eq:apriori-ux}
	\| \p_x u \|_{L^2_x H^1_y} 
	\lesssim 
	\| \p_x f \|_{L^2_x(H^{-1}_y)} +
	\| \Delta_0 \|_{L^2(\Sigma_0, |y| \dd y)} + 
	\| \Delta_1 \|_{L^2(\Sigma_1, |y| \dd y)}.
\end{equation}
However, generally, \textit{the unique weak solution $u\in Z^0$ corresponding to smooth data $(f,\delta_0,\delta_1)$ does not belong to $H^1_x H^1_y$}. 
Indeed, assume that $(f,\delta_0,\delta_1)\in C^\infty_c(\Om)\times C^\infty_c((0,1))\times C^\infty_c((-1,0))$\footnote{The sole purpose of  the regularity and compact support assumptions is to simplify the presentation, and to emphasize that the issue does not lie in the lack of regularity of the data. These assumptions will be removed later.}
, and consider the unique solution $w\in Z^0$ of 
\begin{equation} \label{eq:ux}
	\begin{cases}
		y \p_x w - \p_{yy} w = \p_x f,\\
		w_{\rvert \Sigma_i} = \Delta_i, \\
		w_{\rvert y = \pm 1} = 0.
	\end{cases}
\end{equation}
Attempting to reconstruct $u$ from $w$, we define
\[
\tilde u:=
\begin{cases}
	\delta_0 + \int_{x_0}^x w\quad \text{if } y>0,\\
	\delta_1 - \int_x^{x_1} w\quad \text{if } y<0.
	\end{cases}
\]
Then $\tilde u_{|\Sigma_i}=\delta_i$ by construction, and $\tilde u$ is a solution of \eqref{eq:shear} in $\Om\setminus\{ y=0\}$. However, $\tilde u$ and $\p_y \tilde u$ may have a jump across $y=0$. As a consequence, $\tilde u$ is not a solution of \eqref{eq:shear} in general. More precisely, 
\begin{align*}
&	\tilde u \text{ is a solution of }\eqref{eq:shear} \\
 \iff & [\tilde u]_{|y=0}=  [\p_y \tilde u]_{|y=0}=0\\
	\iff & \delta_0(0)-\delta_1(0)+ \int_{x_0}^{x_1} w(x,0)\:dx= \delta_0'(0)-\delta_1'(0)+\int_{x_0}^{x_1}\p_y  w(x,0)\:dx=0.
\end{align*}
Furthermore, if $ [\tilde u]_{|y=0}=  [\p_y \tilde u]_{|y=0}=0$, then $\tilde u \in Z^0$, and in this case $\tilde u=u$.

Let us now say a few words about the more general case $(f,\delta_0, \delta_1)\in \cH$.
By definition of $\cH$, in this case $\p_x f \in L^2_x H^1_y\subset L^2$, and $\Delta_i= \p_y^2 \delta_i /y \in H^1(|y| \dd y)$. Therefore, according to Proposition \ref{p:pagani-shear}, there exists a unique solution $w\in Z^0$ of \eqref{eq:ux}. 
Following the same argument as above, we define the linear forms
\[
\overline{\ell^j}:(f,\delta_0, \delta_1)\in \cH\mapsto \int_{x_0}^{x_1}\p_y^j w(x,0)\:dx + \p_y^j (\delta_0-\delta_1)(0).
\]
Note that since $w\in Z^0\subset \qnot$ (see \eqref{def:qnot}), $\p_y w_{|y=0}\in H^{1/6}(x_0, x_1)$, and therefore the integral is well-defined.

We eventually obtain the following result:
\begin{proposition}[Orthogonality conditions for higher regularity]~\\
	Let $(f,\delta_0, \delta_1)\in \cH$.
Consider the unique solution $u\in Z^0$ of \eqref{eq:shear}.

Then $\p_x u\in Z^0$  if and only if $\overline{\ell^j}(f,\delta_0, \delta_1)=0$ for $j=0,1$.

\end{proposition}

For further purposes, it is useful to write the linear forms $\overline{\ell^j}$ as scalar products. To that end, let us introduce the following definition:

\begin{definition}[Dual profiles] \label{lem:def-dual-shear}
	We define $\overline{\Phi^0}$, $\overline{\Phi^1} \in Z^0(\Om_\pm)$ as the unique solutions to
	\begin{equation} \label{eq:Phij-shear}
		\begin{cases}
			-y\p_x \overline{\Phi^j} -\p_{yy}\overline{\Phi^j} =0 & \text{in } \Om_\pm,\\
			\left[\overline{\Phi^j} \right]_{|y=0}=\mathbf 1_{j=1}, \\ \left[\p_y\overline{\Phi^j} \right]_{|y=0}=-\mathbf 1_{j=0},\\
			\overline{\Phi^j}_{\rvert \p \Om \setminus(\Sigma_0\cup \Sigma_1)}=0.
		\end{cases}
	\end{equation}
\end{definition}
The existence, uniqueness, and $Z^0$ regularity of $\overline{\Phi^j}$, $j=0,1$, follow easily from the results of Fichera \cite{MR0111931}, Baouendi and Grisvard \cite{BG} and Pagani \cite{Pagani1,Pagani2} recalled above, and from a simple lifting argument. Furthermore, for any
 $f \in H^1((x_0,x_1);L^2(-1,1))$, $\delta_0, \delta_1 \in H^1((-1,1), |y| \dd y)$ with $\delta_0(1) = \delta_1(-1) = 0$ and $\Delta_0, \Delta_1 \in H^1(-1,1, |y| \dd y)$ with $\Delta_0(1) = \Delta_1(-1)= 0$ (see \eqref{eq:def-Di}), it can easily be proved that 
 \[
 \int_{x_0}^{x_1}\p_y^j w(x,0)\:dx = 	\int_\Om \p_x f \overline{\Phi^j} + \int_{\Sigma_0} y\Delta_0 \overline{\Phi^j}  - \int_{\Sigma_1} y \Delta_1 \overline{\Phi^j} ,
 \]
 and therefore
\[
\overline{\ell^j}(f,\delta_0,\delta_1)  = 
\p_y^j \delta_0(0)-\p_y^j\delta_1(0)
+ \int_\Om \p_x f \overline{\Phi^j} 
+ \int_{\Sigma_0} y \Delta_0\overline{\Phi^j}  
- \int_{\Sigma_1} y \Delta_1 \overline{\Phi^j}.
\]   
 Additionally, when the data $(f,\delta_0, \delta_1)$ are smooth and satisfy the orthogonality conditions, we also deduce from the equation a gain of regularity in $y$:
 
 \begin{lemma}\label{prop:dy-5-shearflow}
 	Let $f \in L^2_x H^3_y$  and $\delta_0, \delta_1 \in H^1((-1,1), |y| \dd y)$.
 	Let $u$ be the unique weak solution to \eqref{eq:shear}.
 	Assume that $u \in H^1_x H^2_y$ and $\p_y^3 \delta_i  \in H^1((-1,1), |y| \dd y)$, with
 	$\delta_0(1)=\delta_1(-1)=\Delta_0(1)=\Delta_1(-1)=0$.
 	Assume furthermore that $\p_x \p_y f\in L^2((x_0,x_1)\times (1/2, 1))\cap L^2((x_0,x_1)\times (-1,-1/2))$ and $\Delta_i \in H^2(\Sigma_i\cap \{|y|\geq 1/2\})$.
 	Then $\p_y^5 u \in L^2(\Omega)$ and
 	\begin{equation}
 	\ba 
 			\|u\|_{L^2_x H^5_y} \lesssim &\| u \|_{H^1_x H^2_y} + \| f \|_{L^2_x H^3_y}  + \|\p_x \p_y f \mathbf 1_{|y|\geq 1/2}\|_{L^2} \\
 			& + \sum_{i \in \{0,1\}} \| \p_y^3 \delta_i \|_{H^1((-1,1), |y| \dd y)} + \|\Delta_i\|_{H^2(\Sigma_i\cap \{|y|\geq 1/2\})} .
 		\ea
 	\end{equation}
 \end{lemma}

 In order to conclude this section, let us now investigate the case when the orthogonality conditions  are not satisfied. We first introduce some notation. For $x\in (0, + \infty)$, $y\in \R$, we introduce some ``polar-like'' variables
 \begin{equation} \label{eq:r-t}
 	r := (y^2 + x^{\frac 23})^{\frac 12}
 	\quad \text{and} \quad
 	t := y x^{-\frac 13}
 \end{equation}
 which are consistent with the scaling invariance of the operator $y\p_x - \p_{yy}$. 
 
We first prove the following Lemma:
 \begin{lemma}
  \label{prop:rGk}
For every $k \in \Z$, equation $y\p_x v-\p_{yy} v=0$, set in the domain $(0, +\infty)\times \mathbb R$ and endowed with the boundary condition $v(x=0, y)=0$ for $y>0$,  has a solution of the form 
$v_k = r^{\frac12+3k}\gk{k}(t)$
\begin{equation}
	v_k := r^{\frac 12 + 3k} \gk{k}(t)
\end{equation}
with the variables $(r,t)$ of \eqref{eq:r-t} and $\gk{k} \in C^\infty(\R;\R)$ is a smooth bounded function satisfying $\gk{k}(-\infty) = 1$ and $\gk{k}(+\infty) = 0$.

 \end{lemma}
 
 \begin{proof}[Sketch of proof]
 Using the formulas
 	\begin{align}
 		\label{eq:px-pr-pt}
 		\p_x  & = \frac{(1+t^2)^{\frac12}}{3r^2}\p_r -\frac{t(1+t^2)^{\frac32}}{3r^3}\p_t ,\\
 		\label{eq:pz-pr-pt}
 		\p_y  & = \frac{t}{(1+t^2)^{\frac12}} \p_r  + \frac{(1+t^2)^{\frac12}}{r}\p_t .
 	\end{align}
 	we first obtain an ODE on the function $\gl=\gk{k}$, namely
 		\begin{equation} \label{eq:EDO-G}
 		\d_t^2 \gl(t) + \left( \frac{t^2}{3} + \frac{2 \lambda t}{1+t^2} \right) \d_t \gl(t) + \lambda \left( - \frac{1}{3} \frac{t}{1+t^2} + \frac{1+(\lambda-1) t^2}{(1+t^2)^2} \right) \gl(t) = 0,
 	\end{equation}
 	where $\lambda= \frac{1}{2} + 3k$.
 	Changing variables and setting $\gl=  (1+t^2)^{-\frac \lambda 2} W(-t^3/9)$, we find that $W$ is a solution of Kummer's equation  with $a = - \frac \lambda 3$ and $b = \frac 2 3$:
 	\begin{eqnarray} \label{eq:EDO-W}
 		\zeta \p_\zeta^2 W(\zeta) + \left(\frac 23 - \zeta\right) \p_\zeta W(\zeta) - \left(-\frac \lambda 3\right) W(\zeta) = 0.
 	\end{eqnarray}
 	It is known (see \cite[Section 13.2]{NIST}) that \eqref{eq:EDO-W} has a unique solution behaving like $\zeta^{-a}$ as $\zeta \to \infty$.
 	This (complex valued) solution is usually denoted by $U(a,b,\zeta)$ and called \emph{confluent hypergeometric function of the second kind}, or \emph{Tricomi's function}.
 	In general, $U$ has a branch point at $\zeta = 0$.
 	More precisely, the asymptotic $\zeta^{-a}$ holds in the region $|\arg \zeta| < \frac {3\pi} 2$ and the principal branch of $U(a,b,\zeta)$ corresponds to the principal value of $\zeta^{-a}$.
 	
 	Using explicit formulas for $U(a,b,\zeta)$ and choosing
 	\begin{equation}
 		\label{eq:def-W-zeta}
 		W(\zeta) := \Re \left\{ e^{\frac{i\pi}{3}}  U\left(-\frac\lambda 3,\frac 2 3, \zeta\right) \right\} ,
 	\end{equation}
 	we find that $\gl$ has the desired behavior at $\pm \infty$ iff $\lambda \in \frac{1}{2} + 3 \Z$.
 	\end{proof}

 	The function $v_0$ is linked with a solution to \eqref{eq:shear} which has $Z^0$ regularity, but does not belong to $H^1_x H^1_y$.
 	Similarly, for each $k \geq 0$, $v_k$ is linked with a solution $u$ such that $\p^k_x u \in Z^0(\Omega)$ but $u \notin H^{k+1}_x H^1_y$.
 	  More precisely, we now introduce singular profiles $\busing^i$, for $i=0,1$, localized in the vicinity of $(x_i,0)$.
 Let $\chi_i \in C^\infty(\overline{\Om})$ be a cut-off function such that $\chi_i \equiv 1$ in a neighborhood of $(x_i, 0)$, and $\supp \chi \subset B((x_i, 0), \bar{R})$ for some $\bar{R} < \min (1, x_1-x_0)/2$. 
  
 \begin{definition}[Singular profiles]
 	\label{def:sing-profiles}
 	For $i \in \{0,1\}$, let
 	 	\be
 	\busing^i(x,y) := r_i^{\frac 12} \gnot(t_i) \chi_i(x,y),
 	\ee
 	where $\gnot$ is constructed in Proposition \ref{prop:rGk} and
  	\be\label{def:rj}
 	r_i := \left(y^2 + |x-x_i|^{\frac 23}\right)^{\frac 12}
 	\quad \text{and} \quad
 	t_i := (-1)^i y |x-x_i|^{-\frac 13}.
 	\ee
 \end{definition}
 
 \begin{lemma} \label{lem:busing}
 	For $i \in \{ 0, 1\}$, there exists $\bfi \in C^\infty(\overline{\Om})$, with $\bfi \equiv 0$ in neighborhoods of $(x_i,0)$ and $\{ y = \pm 1 \}$, such that $\busing^i$ is the unique solution with $Z^0(\Omega)$ regularity to
 	\begin{equation}
 		\label{eq:busing}
 		\begin{cases}
 			y \p_x \busing^i - \p_{yy} \busing^i=\bfi,\\
 			\busing^i {}_{|\Sigma_0 \cup \Sigma_1}=0,\\
 			\busing^i {}_{|z=\pm 1}=0.
 		\end{cases}
 	\end{equation}
 	Moreover, $\busing^i \in C^\infty(\overline{\Omega} \setminus \{ (x_i,0) \})$ but $\busing^i \notin H^1_x(H^1_y)$.
 \end{lemma}
\textit{\textbf{Sketch of proof}}
By definition of $\busing^i$, the function $f_i$ is supported in $\supp \nabla \chi_i$. Furthermore there exists $r_\pm>0$ such that $\supp \nabla \chi_i \subset \{ r_-\leq r \leq r_+\}$. Hence the definition of $\bfi$ follows from straightforward computations, and its regularity is a consequence of the  smoothness of $v_0$ away from the origin.

There remains to check that  $\busing^0$, $\p_{yy} \busing^0$ and $z \p_x \busing^0$ are in $L^2(\Omega)$ but $\p_x \p_y \busing^0 \notin L^2(\Omega)$. We use  once again the formulas \eqref{eq:px-pr-pt}-\eqref{eq:pz-pr-pt}, together with the observation that the jacobian of the change of variables $(x,z)\to (r,t)$ is ${(1+t^2)^2}/(3r^3)$. We infer that
\[
y\p_x \busing^i= O( r_i^{-3/2} (1+|t_i|)),
\]
and therefore $y\p_x \busing^i \in L^2$. In a similar fashion $\p_{yy} \busing^i \in L^2$.  However
\[
\p_x \busing^i \in  r_i^{-5/2} H_i(t)  + L^2
\]
for some non-zero function $H_i$, and therefore $\busing^i \notin H^1_x L^2_y$.
\qed

\begin{corollary}[Decomposition into singular profiles]
	\label{coro:decomposition-profile}
	Let     $f\in  H^1_xH^1_y\cap L^2_xH^3_y$ be arbitrary, and let $\delta_0, \delta_1 \in H^1(-1,1)$ with $\delta_0(1)=\delta_1(-1)=0$, and $\Delta_0, \Delta_1 \in H^1(-1,1)$ with $\Delta_0(1)=\Delta_1(-1)=0$.
	Assume furthermore that $\Delta_i \in H^2(\Sigma_i\cap \{|y|\geq 1/2\})$.

	Let $u\in Z^0(\Omega)$ be the unique weak solution to \eqref{eq:shear}. 
	Then there exists two real constants $c_0, c_1$ and a function $\ureg\in \qone$, as defined in \eqref{def:Q1}, such that 
	\be 
	u= c_0 \busing^0 + c_1 \busing^1 + \ureg.
	\ee
\end{corollary}

\begin{proof}
The idea is to adjust the coefficients $c_0$ and $c_1$ in order that the data (source term and boundary conditions) for the remainder $\ureg$ satisfies the orthogonality conditions.
More precisely, we note that $\ureg$ is a solution of
\begin{equation}
	\begin{cases}
		y\p_x \ureg- \p_{yy}\ureg = f-c_0 \bfzero-c_1 \bfun & \text{in }\Om,\\
		u_{\reg|\Sigma_i}=\delta_i,\\
		u_{\reg|y=\pm 1}=0.
	\end{cases}
\end{equation}
We therefore seek $(c_0, c_1)$ so that
\[
\overline{\ell^j} (f-c_0 \bfzero-c_1 \bfun , \delta_0, \delta_1 )=0 \quad \text{for }j=0,1.
\]
Let us first observe that the matrix $\overline{M}:=(\overline{\ell^i} (\bfj, 0,0))_{0\leq i,j\leq 1}$ is invertible. Indeed, if there exists $a=(a_0, a_1)\in \R^2$ such that $\overline{M} a=0$, then 
\[
\overline{\ell^i} (a_0 \bfzero + a_1 \bfun, 0,0)=0\quad \text{for }i=0,1.
\]
It follows from the above arguments that $a_0 \busing^0 + a_1 \busing^1 \in H^1_x L^2_y$. Since the supports of $\busing^0$ and $\busing^1$ are disjoint, we infer that $a_i \busing^i\in H^1_x L^2_y$, and therefore $a_0=a_1=0$. Hence the matrix $\overline{M}$ is invertible.

We then choose the coefficients $(c_0, c_1)$ so that
\[
\overline{M} \begin{pmatrix}
c_0\\ c_1
\end{pmatrix}
= 
\begin{pmatrix}
\overline{\ell^0} (f,\delta_0, \delta_1)\\
\overline{\ell^1} (f,\delta_0, \delta_1)
\end{pmatrix},
\]
which concludes the proof.
\end{proof}

\section{Extension to a broader class of degenerate elliptic equations}


Let us now extend briefly the results of the previous section to equations of the type
\begin{equation} \label{eq:coeff-variables}
	\begin{cases}
		y \partial_x u + \gamma \p_y u- \alpha \partial_{yy} u = f 
		& \text{in } \Omega, \\
		u_{\rvert \Sigma_i} = \delta_i, \\
		u_{\rvert y = \pm 1} = 0.
	\end{cases}
\end{equation}
For the applications we have in mind, the coefficient $\gamma$ will be small (in a norm which will be made precise shortly), and the coefficient $\alpha$ is close to $1$, and therefore keeps a positive sign within $\Om$.
We follow the outline of the previous section: we first state an existence and uniqueness result for weak solutions in $Z^0$. Then, we exhibit orthogonality conditions for higher regularity. 
Eventually, stepping on the analysis of the previous paragraph, we provide a decomposition of the solution associated with smooth data into a sum of singular profiles and a smooth remainder.

\subsection{\texorpdfstring{$Z^0$}{Z0} regularity for weak solutions}

We start with the following result:

\begin{lemma}\label{lem:Z0-coeffs-variables}
 Assume that $\alpha$ and $\gamma$ satisfy
\begin{equation}
	\label{cond-U-alpha-gamma}
	\| \alpha - 1 \|_{L^\infty} 
	+ \|\alpha_y\|_{L^2_y(L^\infty_x)}
	+ \|\gamma\|_{L^2_y(L^\infty_x)} 
	\ll 1.
\end{equation}

Then the following results hold:

\begin{itemize}
\item For every $f\in L^2_x H^{-1}_y$ and $\delta_i \in L^2(\Sigma_i, |y| \dd y)$, there exists a unique weak solution $u\in L^2((x_0,x_1), H^1_0(-1,1))$ to \eqref{eq:coeff-variables}.

\item Let $f\in L^2(\Om)$ and let $\delta_i \in H^1(\Sigma_i, |y| \dd y)$ such that $\delta_0(1)=\delta_1(-1)=0$. Let $u\in L^2_x H^1_y$ be the unique solution of \eqref{eq:coeff-variables}. Then $u\in Z^0$ and
\[
\|u\|_{Z^0} \lesssim C (\| f\|_{L^2} + \|\delta_0\|_{H^1(\Sigma_0, |y|\dd y)} + \|\delta_1\|_{H^1(\Sigma_1, |y|\dd y)} ),
\]
where $C$ is a constant depending only on $\Om$.

\end{itemize}

\end{lemma}
\begin{proof}
The existence of solutions in $L^2_x H^1_y$ follows easily from energy estimates and from the Lax--Milgram lemma, treating the transport and commutator terms perturbatively. Uniqueness is a consequence of the Baouendi--Grisvard theorem \cite{BG}. 

The $Z^0$ estimate is a  bit trickier. Using a compactness argument, it is sufficient to prove the statement for smooth coefficients satisfying \eqref{cond-U-alpha-gamma}. Then, we observe that the main issue is to prove the existence of a $Z^0$ solution. Indeed, if a solution $u\in Z^0$ exists, we write
\[
y \p_x u - \p_{yy} u = f + (\alpha-1)\p_{yy} u - \gamma \p_y u.
\]
We deduce from Proposition \ref{p:pagani-shear} that there exists a universal constant $\bar C$ such that
\[
\ba
\|u\|_{Z^0}\leq &\bar C  \Big(\| f\|_{L^2} + \|\delta_0\|_{H^1(\Sigma_0, |y|\dd y)} + \|\delta_1\|_{H^1(\Sigma_1, |y|\dd y)}  \\
&\qquad+ \|\alpha-1\|_\infty \|u\|_{Z^0} + \|\gamma\|_{L^2_y (L^\infty_x)} \|u\|_{Z^0}\Big).
\ea
\]
Therefore if $\|\alpha-1\|_{\infty} +  \|\gamma\|_{L^2_y (L^\infty_x)} \leq \bar C/4$, we obtain the desired estimate.

Whence it is sufficient to prove the existence of $Z^0$ solutions of \eqref{eq:coeff-variables} for smooth coefficients. We first add a large positive zero order term in the right-hand side. The existence of solutions then follows from \cite[Theorem~5.2]{Pagani2}. We then conclude by a Fredholm type argument. 
\end{proof}

\subsection{Orthogonality conditions for higher order regularity}

We now address the higher regularity theory. In order to simplify the presentation, we focus on the case when the coefficients $\alpha$ and $\gamma$ are smooth. 

We assume that $f\in H^1_x L^2_y$, $\delta_i \in H^2(\Sigma_i)$, and we set
\be
h_i :=\p_x f + \alpha_x \p_{zz}{\delta}_i -\gamma_x \p_y {\delta}_i,\label{def:h_i}
\ee
and
\be\label{def:Delta_i}
{\Delta}_i(y):=\frac{1}{y}\left(f(x_i,y) + \alpha(x_i,y) \p_{yy} {\delta}_i(y) - \gamma(x_i,y) \p_y {\delta}_i(y)\right).
\ee

Let $u\in Z^0$ be the solution of \eqref{eq:coeff-variables}. Assume that $\p_x u \in Z^0$. Then, differentiating \eqref{eq:coeff-variables} with respect to $x$ and setting $\Om_\pm:=\Om \cap \{y\gtrless 0\}$, we observe that $v=\p_x u$ is a solution of
\be\label{eq:V}
\begin{cases}
	y \p_x v
	+ \gamma \p_y v
	- \alpha \p_{yy} v -\alpha_x \p_{yy} \int_{x_0}^x v+ \gamma_x \p_y\int_{x_0}^x v  = h_0 & \text{in } \Om_+,\\
	y \p_x v
	+ \gamma \p_y v
	- \alpha \p_{yy} v +\alpha_x \p_{yy} \int^{x_1}_x v - \gamma_x \p_y\int^{x_1}_x v    = h_1 & \text{in } \Om_-
	, \\
	[v]_{y=0}=[\p_y v]_{y=0}=0 & \text{on } (x_0,x_1),\\
	v(x_0,y) = {\Delta}_0 & \text{for } y\in (0,1), \\
	v(x_1,y) = {\Delta}_1 & \text{for } y \in (-1,0), \\
	v(x,\pm 1) = 0 & \text{for } x \in (x_0,x_1),
\end{cases}
\ee
Reciprocally, assuming that the above system has a unique weak solution $v\in Z^0$. Then $u_1$ defined by
\be
u_1 := 
\begin{cases}
	{\delta}_0 + \int_{x_0}^x v & \text{in } \Om_+, \\
	{\delta}_1 + \int_{x_1}^x v & \text{in } \Om_-
\end{cases}
\ee
is a solution to \eqref{eq:U} if and only if $v$ satisfies 
\be\label{cond-V}
\ba
\int_{x_0}^{x_1} v(x,0)\dd x & = {\delta}_1(0)-{\delta}_0(0),\\
\int_{x_0}^{x_1} \p_y v(x,0)\dd x & = \p_y{\delta}_1(0)-\p_y {\delta}_0(0).
\ea
\ee
As in the previous section, we find that $\p_x u\in Z^0$ if and only if the data $(f,\delta_0,\delta_1)$ satisfy two orthogonality conditions.
More precisely, we prove the following result:

\begin{proposition}
 Let $\alpha\in H^1_x H^1_y \cap H^{3/5}_x H^2_y$.
Assume that there exists $\gamma_1\in H^{2/3}_xL^2_y\cap L^\infty_y(H^{1/2}_x) \cap W^{1,\infty}_y(L^2_x)$ and $\gamma_2\in H^1_x L^2_y \cap H^{3/5}_x H^1_y$ such that $\gamma=y\gamma_1 + \gamma_2$ and
\begin{equation}
	\label{eq:smallness-ag1g2}
	\begin{split}
		\|\alpha-1\|_{H^1_x H^1_y}  +\|\p_{yy} \alpha\|_{H^{3/5}_x L^2_y} & \ll 1,\\
		\|\gamma_1\|_{H^{2/3}_xL^2_y} + \|\gamma_1\|_{L^\infty_y(H^{1/2}_x)} + \| \p_y\gamma_1\|_{L^\infty_y(L^2_x)} & \ll 1,\\
		\|\gamma_2\|_{H^1_xL^2_y} + \|\p_y\gamma_2 \|_{H^{3/5}_x L^2_y} & \ll 1.
	\end{split}
\end{equation}

There exist two independent linear forms $\ell_{\alpha,\gamma}^0, \ell_{\alpha,\gamma}^1$, defined on $\cH$, such that the following result holds. 
Let $(f,\delta_0,\delta_1)\in \cH$.
Assume that 
${\delta_0}(1)={\delta_1}(-1)=0$, and that $\Delta_i \in L^2(\Sigma_i, |y|\dd y)$.

Let $u\in Z^0(\Omega)$ be the unique solution to \eqref{eq:U}. Then $u\in H^1_xH^1_y$ if and only if
\[
\ell_{\alpha,\gamma}^0(f,\delta_0,\delta_1)=\ell_{\alpha,\gamma}^1(f,\delta_0,\delta_1)=0,
\]
and in this case
\[
\|u\|_{H^1_xH^1_y}\lesssim \|f\|_{H^1_xL^2_y} + \|{\delta_i}\|_{H^2} + \| \Delta_i \|_{L^2(\Sigma_i, |y| \dd y)}.
\]

\label{prop:U-recap}

\end{proposition}
\begin{proof}[Sketch of proof]
When the coefficients $\alpha$ and $\gamma$ are smooth, the result is a straightforward consequence of the argument above. Therefore the issue is to extend both the definition of the linear forms and the notion of weak solution of \eqref{eq:V} to coefficients $\alpha,\gamma$ which merely have the regularity stated in the Proposition.
We argue by duality, identifying in particular the ``dual profiles'' $\Phi^j$ associated with the linear forms $\ell^j_{\alpha,\gamma}$. 
We refer to \cite{DMR} for all details.

\end{proof}

Furthermore, the linear forms $\ell^j_{\alpha,\gamma}$ depend smoothly on the coefficients:

\begin{lemma}
	\label{lem:regularity-phi-j}
	Let $(\alpha,\gamma)$ and $(\alpha',\gamma')$ be two sets of coefficients satisfying the assumptions of Proposition \ref{prop:U-recap}. Then
	\be
	\| \ell_{\alpha,\gamma}^j - \ell_{\alpha',\gamma'}^j \|_{\mathcal L(\mathcal H)}
		\lesssim \|\alpha-\alpha'\|_{L^\infty_y(H^{7/12}_x)} + \|\gamma_1-\gamma_1'\|_{L^\infty_y(L^2_x)} + \|\gamma_2-\gamma_2'\|_{H^{1/2}_x L^2_y}.
	\ee

\end{lemma}

\begin{proof}[Sketch of proof]
We write the equation satisfied by the dual profiles $\Phi^j$ and $(\Phi^j)'$. Their difference satisfies an equation of the form \eqref{eq:coeff-variables} (with an additional nonlocal transport term and variable coefficients), and with a source term involving $\alpha-\alpha'$, $\gamma-\gamma'$. Performing energy estimates, we prove that
\[
	\|\Phi^j-(\Phi^j)'\|_{L^2_x H^1_y(\Om_\pm)} 
	\lesssim \|\alpha-\alpha'\|_{L^\infty_y(H^{7/12}_x)} + \|\gamma_1-\gamma_1'\|_{L^\infty_y(L^2_x)} + \|\gamma_2-\gamma_2'\|_{H^{1/2}_x L^2_y}.
\]
The result follows.
\end{proof}

\section{The nonlinear scheme}

We now turn towards the proof of Theorem \ref{theorem:main}. 
Uniqueness follows easily from simple energy estimates and is left to the reader. Therefore we focus on the existence part of the Theorem.

We construct an iterative sequence $(u_n)_{n\in \N}$ in the following way.

Let $\chi \in C^\infty(\mathbb{R},[0,1])$, identically equal to one on $[-\frac 13, \frac 13]$ and compactly supported in $[-\frac 12, \frac 12]$. We define the initialization profile of our iterative scheme as
\begin{equation} \label{eq:u0}
	u_0(x,y) := \delta_0(y) \chi \left(\frac{x - x_0}{x_1-x_0}\right) + \delta_1(y) \chi \left(\frac{x_1 - x}{x_1-x_0}\right).
\end{equation}
Then, for any $n\geq 0$, we define $u_{n+1}$ as the solution of 
\be
\label{eq:NL-un-un1}
\begin{cases}
	(y + u_n) \p_x u_{n+1} - \p_{yy} u_{n+1} = f+ \nu^0_{n+1} f^0 + \nu^1_{n+1} f^1=:f_{n+1}, \\
	(u_{n+1})_{\rvert \Sigma_i} = \delta_i + \nu^0_{n+1} \delta^0_i + \nu^1_{n+1} \delta^1_i=:\delta_{i,n+1}, \\
	(u_{n+1})_{\rvert y = \pm 1} = 0,
\end{cases}
\ee
In the above equation  $(f^j, \delta^j_0, \delta^j_1)$ are fixed triplets (independent of $n$) such that
\be\label{def:Xi-j}
\overline{\ell^i}(f^j, \delta^j_0, \delta^j_1)=\delta_{i,j}.
\ee
The coefficients $\nu^0_{n+1}$ and $\nu^1_{n+1}$ are designed to ensure that the orthogonality conditions are satisfied at every step.

The main steps of the analysis are the following:
\begin{itemize}
\item First, we perform a change of variables in order to transform \eqref{eq:NL-un-un1} into an equation of the form \eqref{eq:coeff-variables} with coefficients $\alpha_n,\gamma_n$ depending on $u_n$. This will allow us to define rigorously the coefficients $\nu^j_n$, and therefore the sequence $(u_n)_{n\in \N}$.

\item Then, we derive uniform estimates in $\qone$ on the sequence $(u_n)_{n\in \N}$. This step is fairly easy and is a straightforward consequence of the regularity analysis of solutions of equation \eqref{eq:coeff-variables}.

\item Eventually, we prove that $(u_n)_{n\in \N}$ is a Cauchy sequence in a suitable functional space. This is the most involved step of the convergence proof.
Indeed, the sequence $(u_n)_{n\in \N}$ is not sufficiently smooth in order to obtain a Cauchy bound in $\qone$. Hence we must work in a space with lower regularity.
However, the Lipschitz estimate on the linear forms $\ell^j_{\alpha_n,\gamma_n}$ from Lemma \ref{lem:regularity-phi-j} involves a regularity which is strictly bigger than $\qnot$.
This prompts us to work in a fractional Sobolev space. 
\end{itemize}

\subsection{Change of variables and uniform \texorpdfstring{$\qone$}{Q1} estimate}
Let $n\in \N$, and assume that $u_n\in \qone$ is such that $\|u_n\|_{\qone}\ll 1$.
Define $Y_n=Y_n(x,z)$ by the implicit formula
\[
Y_n(x,z) + u_n(x, Y_n(x,z))=z\quad \forall (x,z)\in \overline{\Om}.
\]
We then look for $u_{n+1}$ under the form
\begin{equation}
	u_{n+1}(x,y) = U_{n+1}(x, y+ u_n(x,y)),
\end{equation}
so that $U_{n+1}=U_{n+1}(x,z)$ solves
\begin{equation} \label{eq:U}
	\begin{cases}
		z \partial_x U_{n+1} + \gamma_n \p_z U_{n+1} - \alpha_n \partial_{zz} U_{n+1} = g_{n+1} 
		& \text{in } \Omega, \\
		U_{n+1\rvert \Sigma_i} = \widetilde{\delta}_{i,n+1}, \\
		U_{n+1\rvert y = \pm 1} = 0,
	\end{cases}
\end{equation}
where
\begin{align}
	\label{eq:alpha}
	\alpha_n(x,z) & := (1 + \p_y u_n)^2(x,Y_n(x,z)), \\
	\label{eq:gamma}
	\gamma_n(x,z) & := (z \p_x u_n- \p_{yy} u_n)(x,Y_n(x,z)) \\
	\label{eq:g}
	g_{n+1}(x,z) & := f_{n+1}(x, Y_n(x,z)), \\
	\label{def:tilde-delta}
	\widetilde{\delta}_{i,n+1}(z) & := \delta_{i,n+1} (Y_n(x_i,z)).
\end{align}
It follows from Lemma \ref{lem:Z0-coeffs-variables} that equation \eqref{eq:U} has a unique solution for every couple $(\nu^0_{n+1}, \nu^1_{n+1})\in \R^2$. 
Furthermore, it can be checked that $\alpha_n, \gamma_n$ satisfy the assumptions of \eqref{prop:U-recap}.
With a slight abuse of notation, let us denote by $\ell^j_n$ the linear form $\ell^j_{\alpha_n,\gamma_n}$. 
We then choose the coefficients $(\nu^0_{n+1}, \nu^1_{n+1})\in \R^2$ so that
\[
\ell^j_n (g_{n+1}, \widetilde{\delta}_{0,n+1},\widetilde{\delta}_{1,n+1})=0\quad \text{for }j=0,1.
\]
Using Lemma \ref{lem:regularity-phi-j}, we prove that the matrix 
\[
M_n:=\left( \ell^j_n ( f^i(x, Y_n), \delta^i_0(Y_n(x_0)), \delta^i_1(Y_n(x_1)))\right)_{0\leq i, j \leq 1}
\]
is close to identity, and thus invertible. The coefficients $(\nu^0_{n+1}, \nu^1_{n+1})$ are therefore uniquely defined.

Furthermore, using Proposition \ref{prop:U-recap} and deriving additional regularity estimates in the spirit of Lemma \ref{prop:dy-5-shearflow}, we prove by induction that
\[
\| u_n\|_{\qone}  + | \nu^0_n | + |\nu^1_n| \lesssim \|(f,\delta_0, \delta_1)\|_{\cH} \quad \forall n\in \N.
\]

\subsection{Geometric estimate in \texorpdfstring{$\qhalf$}{Q12}}

The next step is to prove a bound of the type
\be\label{in-geom}
\|u_{n+1}- u_n \|_X \lesssim \eta \| u_n - u_{n-1}\|_X
\ee
for some $0<\eta \ll 1$, in a functional space $X$ to be determined. To that end, we observe that $w_n:=u_{n+1}- u_n$ satisfies
\begin{equation} \label{eq:wn}
	\begin{cases}
		(y + u_n) \partial_x w_n - \partial_{yy} w_n = - w_{n-1} \partial_x u_n + (\nu^0_{n+1}-\nu^0_{n}) f^0 + (\nu^1_{n+1}-\nu^1_{n}) f^1, \\
		(w_n)_{\rvert \Sigma_i} = (\nu^0_{n+1} - \nu^0_n) \delta^0_i + (\nu^1_{n+1} - \nu^1_n) \delta^1_i, \\ 
		(w_n)_{\rvert y = \pm 1} = 0.
	\end{cases}
\end{equation}
We already know that the solution $w_n$ belongs to $\qone$, as the difference between two $\qone$ functions.
Hence, there is no need to check that the orthogonality conditions are satisfied.
However, because of the term $-w_{n-1} \partial_x u_n$ in the right-hand side, the source term does not belong to $H^1_x L^2_y$. 
Note that the right-hand side also involves $\nu^j_{n+1} - \nu^j_n$. From the definition of $\nu^j_n$ and Lemma \ref{lem:regularity-phi-j}, we have
\[
\ba 
|\nu^j_{n+1} - \nu^j_n| \lesssim &\| \ell^0_{n}- \ell^0_{n-1}\|_{\mathcal L(\cH)} +  \| \ell^1_{n}- \ell^1_{n-1}\|_{\mathcal L(\cH)} \\
\lesssim & \|\alpha_n-\alpha_{n-1}\|_{L^\infty_z(H^{7/12}_x)} + \|\gamma_{1,n}-\gamma_{1, n-1} \|_{L^\infty_z(L^2_x)} \\&+ \|\gamma_{2, n} -\gamma_{2, n-1}\|_{H^{1/2}_x L^2_z}.
\ea 
\]
Looking at the definition of the coefficients $\alpha_n, \gamma_n$, one can prove that the right-hand side of the above inequality is bounded by $\| u_n - u_{n-1}\|_{\qhalf}$. This prompts us to take 
\[
X=Q^{1/2}:= H^{7/6}_x L^2_y \cap L^2_x H^{7/2}_y(\Om).
\]
in \eqref{in-geom}.

As a consequence, we need to derive estimates in $\qhalf$ for the solutions of \eqref{eq:coeff-variables}, when the data satisfies an orthogonality condition.
We choose to argue by interpolation. This leads to a subtle matter:  omitting the boundary data to simplify the discussion, one needs to interpolate between $L^2(\Om)$ (which is the regularity of the source term for $\qnot$ solutions) and the space $\{f\in H^1_x L^2_z \cap L^2_x H^3_z\ \ell^j_{\alpha,\gamma} (f,0,0)=0\}$.
This requires a rather precise description of the linear forms $\ell^j$, involving in particular their representation in terms of the dual profiles $\Phi^j$, and a decomposition of the latter into singular profiles and a smooth and/or explicit remainder, in the spirit of Corollary \ref{coro:decomposition-profile}. We refer to \cite{DMR} for all details. 
Eventually, we conclude that
\[
\|w_n \|_{\qhalf} \lesssim \eta \|w_{n-1}\|_\qhalf + \eta^2 \|w_{n-2}\|_\qhalf,
\]
where $\eta= \|(f,\delta_0, \delta_1)\|_{\cH} $.

As a consequence, $(u_n)_{n\in \N}$ is a Cauchy sequence in $\qhalf$. Passing to the limit as $n\to \infty$, we infer that there exists $u\in \qone$, $\|u\|_{\qone}\lesssim \eta$, and $(\nu^0, \nu^1)\in \R^2$, $| \nu^j|\lesssim \eta$, such that
\[
\begin{cases}
(y+u)\p_x u - \p_{yy} u =f + \nu^0 f^0 + \nu^1 f^1,\\
	u_{\rvert \Sigma_i} = \delta_i + \nu^0 \delta^0_i + \nu^1 \delta^1_i, \\
u_{\rvert y = \pm 1} = 0.
\end{cases}
\]
Let $B_\eta$ be the open ball of radius $\eta$ in $\cH$. Note that it follows from the above construction that the maps
\[
(f,\delta_0, \delta_1)\in B_\eta\mapsto \nu^j (f,\delta_0, \delta_1)\in \R
\]
are Lipschitz continuous. We also denote by $\mathcal{U}_\perp$ the map $\mathcal{U}_\perp:(f,\delta_0, \delta_1)\in B_\eta\mapsto u\in \qone$. The above argument shows that $\mathcal{U}_\perp$ is Lipschitz continuous from $B_\eta$ to $\qhalf$.

\subsection{Definition of the manifold \texorpdfstring{$\mathcal M$}{M}}

For the sake of brevity, let us denote by $\Xi=(f,\delta_0,\delta_1)$ the elements of $\cH$. We set  $\Xi^j= (f^j, \delta_0^j, \delta_1^j)$ (see \eqref{def:Xi-j}), and $\hperp:=\{\Xi\in \cH, \ \langle \Xi, \Xi^j\rangle_{\cH} =0\ j=0,1 \}$.
For every $\Xi  \in \mathcal{H}$, one has the decomposition
\begin{equation}
	\Xi = \Xi^\perp + \langle \Xi^0 ; \Xi \rangle_{\mathcal{H}} \Xi^0 + \langle \Xi^1 ; \Xi \rangle_{\mathcal{H}} \Xi^1,
\end{equation}
where $\Xi^\perp \in \hperp$ and the linear maps $\Xi \mapsto \Xi^\perp$ and $\Xi \mapsto \langle \Xi^k ; \Xi \rangle$ are continuous.

Let us now define
 \begin{equation} \label{eq:M}
	\mathcal{M} := \left\{ 
	\Xi \in \mathcal{H}; \enskip
	\| \Xi \|_{\mathcal{H}} < \eta 
	\text{ and }
	\langle \Xi^k ; \Xi \rangle_{\mathcal{H}} = \nu^k(\Xi^\perp)
	\text{ for } k = 0,1 
	\right\}
\end{equation}
where $\nu^0$ and $\nu^1$ are the maps constructed in the previous paragraph.
We set, for $\Xi \in \mathcal{M}$,
\begin{equation}
	\mathcal{U}(\Xi) := \mathcal{U}_\perp(\Xi^\perp).
\end{equation}
It can be easily checked that for all $\Xi\in \mathcal{M}$, $\mathcal{U}(\Xi)$ is a solution of \eqref{eq:eq0}. It is also clear from the definition that $\mathcal M$ is a Lipschitz manifold modeled on $\hperp$ since $\nu^0$ and $\nu^1$ are Lipschitz maps. It contains $0$ since $\nu^j(0)=0$. Furthermore, $\hperp$ is  tangent to $\mathcal{M}$ at $0$ in the following weak senses:
\begin{itemize}
	\item For $\Xi \in \mathcal{M}$, $\operatorname{d} ( \Xi, \hperp ) \lesssim \| \Xi \|^2_{\mathcal{H}}$.
	\item For every $\Xi^\bot \in \hperp$, for $\eps \in \R$ small enough, $\operatorname{d} (\eps \Xi^\bot, \mathcal{M}) \lesssim \eps^2$.
\end{itemize}
These two properties are left to the reader. Theorem \ref{theorem:main} follows.

\section*{Acknowledgements}

This project has received funding from the European Research Council (ERC) under the European Union's Horizon 2020 research and innovation program Grant agreement No 637653, project BLOC ``Mathematical Study of Boundary Layers in Oceanic Motion''. This work was supported by the SingFlows project, grants ANR-18-CE40-0027  of the French National Research Agency (ANR).
A.-L.\ D.\ acknowledges the support of the Institut Universitaire de France.
This material is based upon work supported by the National Science Foundation under Grant No. DMS-1928930 while A.-L.\ D.\ participated in a program hosted by the Mathematical Sciences Research Institute in Berkeley, California, during the Spring 2021 semester.

	%
	%
\end{document}